\documentclass[a4paper, 12pt]{article}

\usepackage[top = 2.5cm, bottom = 2.5cm, left = 2.4cm, right = 2.4cm]{geometry}
\usepackage{amsfonts, amsmath, amsthm, amssymb, mathtools, cases, bbm, array}
\usepackage{lmodern, indentfirst, setspace, changepage, fancyhdr}
\usepackage{tikz, pgfplots, float, subcaption}
\usepackage[title]{appendix}
\usepackage[english]{babel}
\usepackage{cite}

\usepackage{hyperref}
\hypersetup{
	colorlinks = true,
	linkcolor = blue,
	citecolor = blue,
}

\newtheorem{proposition}{Proposition}

\newtheorem{theorem}{Theorem}
\newtheorem{lemma}{Lemma}
\theoremstyle{definition}

\newtheorem{definition}{Definition}
\newtheorem{example}{Example}
\newtheorem{remark}{Remark}

\newcommand{\ind}[1]{\mathbbm{1}_{\left\{#1\right\}}}

\newcommand{\ceil}[1]{\left\lceil#1\right\rceil}
\newcommand{\map}[3]{#1 : #2 \longrightarrow #3}
\newcommand{\set}[2]{\left\{#1 : #2\right\}}

\newcommand{\sett}[2]{\left(#1 : #2\right)}

\newcommand{\defeq}{\vcentcolon=}

\newcommand{\bq}{\boldsymbol{q}}

\newcommand{\bX}{\boldsymbol{X}}
\newcommand{\bY}{\boldsymbol{Y}}

\newcommand{\calN}{\mathcal{N}}

\newcommand{\calS}{\mathcal{S}}

\newcommand{\N}{\mathbbm{N}}

\newcommand{\e}{\mathrm{e}}

\newcommand{\expect}{E\expectarg}

\DeclarePairedDelimiterX{\expectarg}[1]{[}{]}{%
	\ifnum\currentgrouptype=16 \else\begingroup\fi
	\activatebar#1
	\ifnum\currentgrouptype=16 \else\endgroup\fi
}
\newcommand{\cprob}{P\probarg}
\DeclarePairedDelimiterX{\probarg}[1]{(}{)}{%
	\ifnum\currentgrouptype=16 \else\begingroup\fi
	\activatebar#1
	\ifnum\currentgrouptype=16 \else\endgroup\fi
}
\newcommand{\innermid}{\nonscript\;\delimsize\vert\nonscript\;}
\newcommand{\activatebar}{%
	\begingroup\lccode`\~=`\|
	\lowercase{\endgroup\let~}\innermid 
	\mathcode`|=\string"8000
}

\usetikzlibrary{automata, arrows, positioning, calc, external, babel, backgrounds, matrix, shapes}
\usepgfplotslibrary{fillbetween}
\pgfplotsset{
	compat = 1.16,
	ticklabel style = {font = \footnotesize},
	every axis/.append style = {
		grid style = {dashed, gray, opacity = 0.2},
		label style = {font = \footnotesize}, 
		width = \columnwidth,
		height = 0.618 * 1 * \columnwidth
	}
}

\definecolor{britishracinggreen}{rgb}{0.0, 0.26, 0.15}
\definecolor{bostonuniversityred}{rgb}{0.8, 0.0, 0.0}
\definecolor{ceruleanblue}{rgb}{0.16, 0.32, 0.75}
\definecolor{airforceblue}{rgb}{0.36, 0.54, 0.66}
\definecolor{cadmiumgreen}{rgb}{0.0, 0.42, 0.24}
\definecolor{ao(english)}{rgb}{0.0, 0.5, 0.0}
\definecolor{coolblack}{rgb}{0.0, 0.18, 0.39}
\definecolor{byzantine}{rgb}{0.74, 0.2, 0.64}
\definecolor{alizarin}{rgb}{0.82, 0.1, 0.26}
\definecolor{arsenic}{rgb}{0.23, 0.27, 0.29}
\definecolor{cobalt}{rgb}{0.0, 0.28, 0.67}
\definecolor{amber}{rgb}{1.0, 0.75, 0.0}

\title{Geometric lower bounds for the steady-state occupancy\\ of processing networks with limited connectivity \vspace{\baselineskip}}

\author{
	\begin{tabular}{ccc}
		\normalsize Diego Goldsztajn & \hspace{1cm} & \normalsize Andres Ferragut \\ 
		\small Universidad ORT Uruguay & \hspace{1cm} & \small Universidad ORT Uruguay \\
		\scriptsize\texttt{goldsztajn@ort.edu.uy} & \hspace{1cm} & \scriptsize\texttt{ferragut@ort.edu.uy} \\
	\end{tabular}
}

\date{\vspace{\baselineskip} April 29, 2026}

\begin{document}

	
\maketitle

\noindent\rule{\textwidth}{1pt}

\vspace{2\baselineskip}

\onehalfspacing

\begin{adjustwidth}{0.8cm}{0.8cm}
	\begin{center}
		\textbf{Abstract}
	\end{center}
	
	\vspace{0.3\baselineskip}
	
	\noindent	
	We consider processing networks where multiple dispatchers are connected to single-server queues by a bipartite compatibility graph, modeling constraints that are common in data centers and cloud networks due to geographic reasons or data locality issues. We prove lower bounds for the steady-state occupancy, i.e., the complementary cumulative distribution function of the empirical queue length distribution. The lower bounds are geometric with ratios given by two flexibility metrics: the average degree of the dispatchers and a novel metric that averages the minimum degree over the compatible dispatchers across the servers. Using these lower bounds, we establish that the asymptotic performance of a growing processing network cannot match that of the classic Power-of-$d$ or JSQ policies unless the flexibility metrics approach infinity in the large-scale limit. 
	
	\vspace{\baselineskip}
	
	\small{\noindent \textit{Key words:} load balancing in networks, compatibility graph, geometric lower bounds.}
	
	\vspace{0.3\baselineskip}
	
	\small{\noindent The research presented in this paper was supported by ANII-Uruguay under fellowship PD\_NAC\_2024\_1\_182118 and by AFOSR US under grant FA9550-23-1-0350.} 
\end{adjustwidth}

\newpage


\section{Introduction}
\label{sec: introduction}

Load balancing algorithms play a central role in parallel-processing systems such as data centers and cloud networks, where they distribute the incoming user requests or tasks across the servers. Well-designed load balancing strategies can drastically reduce the delay experienced by users through efficient resource utilization, which has attracted immense attention in the past few decades;  an extensive survey is provided in \cite{van2018scalable}.

Until recently, this attention had focused on the so-called \emph{supermarket model}, where a single dispatcher distributes the incoming tasks across parallel single-server queues that resemble checkout counters. The steady-state performance of this system can be evaluated through the steady-state occupancy $q$, which is the random sequence such that $q(i)$ is the fraction of servers with at least $i$ tasks in steady state. Under exponential assumptions, $q$ has been analyzed for different load balancing policies in the limit as the number of servers and the arrival rate of tasks approach infinity, while the load $\rho < 1$ remains fixed.

For the na\"{i}ve policy that dispatches each task uniformly at random, the limit $q$ of the steady-state occupancy is the deterministic geometric sequence with ratio $\rho$. However, if $d \geq 2$ queues are selected uniformly at random, instead of just one, and the task is sent to the shortest of these $d$ queues, then performance drastically improves since $q$ decays doubly-exponentially instead of geometrically; this remarkable power-of-$d$ property was established in \cite{vvedenskaya1996queueing,mitzenmacher2001power}. The Join-the-Idle-Queue (JIQ) and Join-the-Shortest-Queue (JSQ) policies perform even better since $q(1) = \rho$ and $q(i) = 0$ for all $i \geq 2$; see \cite{lu2011join,stolyar2015pull,mukherjee2018universality}. The JSQ policy is not only natural but also optimal in the pre-limit under exponential assumptions, as shown in \cite{menich1991optimality,sparaggis1993extremal}. Yet, the Power-of-$d$ and JIQ policies are better suited for systems with many servers, where JSQ incurs an important communication overhead.

Recently, the spotlight has shifted from the supermarket model to models that capture the distributed nature of cloud networks and their data locality issues. In particular, tasks cannot be dispatched to any server in the network, but only to the subset of servers that are sufficiently near and have preinstalled the data or machine learning model required to complete the task. Such \emph{compatibility constraints} were first modeled by a graph connecting the servers in \cite{gast2015power,mukherjee2018asymptotically,budhiraja2019supermarket} and then by a bipartite graph between dispatchers and servers in \cite{weng2020boptimal,rutten2022load,rutten2024mean,zhao2022exploiting,zhao2023optimal}, where each dispatcher represents a type of task. The focus has been on finding connectivity conditions for a sequence of graphs or bipartite graphs such that the limitng steady-state occupancy is as for Power-of-$d$ or JSQ in the supermarket model. 

In this paper we consider the bipartite graph model, which is more general and versatile, and we prove results in the converse direction: instead of deriving sufficient conditions for a near-ideal limiting performance, we obtain \emph{necessary conditions}. Our results are based on novel \emph{flexibility metrics} that capture the diversity in how dispatchers are connected to servers. We derive lower bounds for the steady-state occupancy that hold in the pre-limit and depend geometrically on these metrics. We further prove that these lower bounds hold asymptotically for sequences of networks where any of the flexibility metrics has a finite limit. Therefore, it is necessary that the flexibility metrics diverge to match the limiting performance of the classic Power-of-$d$ or JSQ policies. One of the flexibility metrics is the average degree of the dispatchers, which is assumed to diverge in \cite{weng2020boptimal,rutten2022load,rutten2024mean,zhao2022exploiting,zhao2023optimal, mukherjee2018asymptotically, budhiraja2019supermarket}. Our results rigorously prove that this assumption is actually necessary for the latter results.

The remainder of the paper is organized as follows. In Section \ref{sec: model description} we describe the bipartite graph model formally. In Section \ref{sec: main results} we introduce the flexibility metrics and  state our main results. In Section \ref{sec: monotone transformations} we define monotone network transformations  that are used in the proofs of the main results. In Sections \ref{sec: first lower bound} and \ref{sec: second lower bound} we prove the main results. A brief conclusion is given in Section \ref{sec: conclusion} and some intermediate results are proved in Appendix \ref{app: proofs of several results}.

\section{Model description}
\label{sec: model description}

We consider networks represented by bipartite graphs $G = (D, S, E)$. The finite sets $D$ and $S$ represent dispatchers and servers, respectively, and compatibility constraints are encoded in $E \subset D \times S$, i.e., dispatcher $d$ can only send tasks to server $u$ if $(d, u) \in E$. Naturally, we assume that each dispatcher is compatible with at least one server and each server is compatible with at least one dispatcher, i.e., $G$ does not have isolated nodes. The dispatchers need not represent actual load balancing devices, but rather correspond to specific streams of tasks that are distributed among common subsets of servers. In practice, the servers that are compatible with a given task are those that are geographically close to the user that requested the task and have the data required to perform the task.

We assume that each dispatcher $d$ receives tasks as an independent Poisson process of intensity $\lambda(d)$, and consider the natural load balancing policy such that each dispatcher sends every incoming task to the compatible server with the shortest queue, with ties broken uniformly at random. Tasks are executed sequentially at each server $u$ and have independent and exponentially distributed service times with rate $\mu(u)$. If we let $\bX(t, u)$ be the number of tasks in server $u$ at time $t$, then the latter assumptions imply that the stochastic process $\bX$ is a continuous-time Markov chain with values in $\N^S$.

\begin{definition}
	\label{def: load balancing process}
	We say that $\bX$ is the \emph{load balancing process} associated with the bipartite graph $G = (D, S, E)$ and the rate functions $\map{\lambda}{D}{(0, \infty)}$ and $\map{\mu}{S}{(0, \infty)}$.
\end{definition}

We are interested in the stationary behavior of load balancing processes. A sufficient condition for the ergodicity of such processes can be obtained from \cite[Theorem 2.5]{foss1998stability}. Specifically, let $\bX$ be as before and let $\calN(d) \defeq \set{u \in S}{(d, u) \in E}$ denote the set of servers that are compatible with some dispatcher $d$. Then the condition 
\begin{equation}
	\label{eq: ergodicity condition}
	\sum_{\calN(d) \subset U} \lambda(d) < \sum_{u \in U} \mu(u) \quad \text{for all} \quad \emptyset \neq U \subset S
\end{equation}
implies that $\bX$ is ergodic. Conversely, suppose that there exists some $U \subset S$ such that the strict inequality holds in the opposite direction. Then \cite[Theorem 2.7]{foss1998stability} implies that the process $\bX$ is not ergodic. Moreover, if $\set{\tau_k}{k \geq 1}$ denote the arrival times of tasks to the system, then the following instability property holds with probability one:
\begin{equation*}
	\liminf_{k \to \infty} \frac{1}{k}\sum_{u \in U} \bX(\tau_k, u) > 0.
\end{equation*}

One of the simplest load balancing processes is that where a single dispatcher distributes tasks across servers with the same service rate, as defined below.

\begin{definition}
	\label{def: simple load balancing process}
	Consider a network $G = (D, S, E)$ with $D$ a singleton and $E = D \times S$, and let the rate functions $\lambda$ and $\mu$ be constant. The load balancing process associated with these bipartite graph and rate functions is called \emph{simple} and its load is $\rho \defeq \lambda / \mu$.
\end{definition}

The simple load balancing process corresponds to the standard supermarket model where the dispatcher uses the JSQ policy, and is ergodic if and only if $\rho < |S|$.

\subsection{Steady-state occupancy}

The performance of a load balancing system is typically evaluated by considering the occupancy process $\bq$ associated with $\bX$. Namely, we define
\begin{equation*}
	\bq(t, i) \defeq \frac{1}{|S|}\sum_{u \in S} \ind{\bX(t, u) \geq i} \quad \text{for all} \quad t \geq 0 \quad \text{and} \quad i \in \N,
\end{equation*}
which represents the fraction of servers with at least $i$ tasks at time $t$. In general, $\bq$ is not a Markov process because $\bq(t)$ contains less information than $\bX(t)$ about the state of the system at time $t$. However, we may define the steady-state occupancy as follows.

\begin{definition}
	\label{def: occupancy state}
	Suppose that $\bX$ is an ergodic load balancing process and let $X$ denote its stationary distribution. The \emph{steady-state occupancy} is the random sequence defined as
	\begin{equation*}
		q(i) \defeq \frac{1}{|S|}\sum_{u \in S} \ind{X(u) \geq i} \quad \text{for all} \quad i \in \N.
	\end{equation*} 
\end{definition}

It is clear that a system for which $q(i)$ is smaller has fewer queues of length larger than $i$ in steady state, and therefore offers a better performance. Moreover,
\begin{equation*}
	|S|\sum_{i = 1}^\infty q(i) = \sum_{i = 1}^\infty i|S|\left[q(i) - q(i + 1)\right]
\end{equation*}
is the steady-state total number of tasks in the system; its mean is proportional to the mean sojourn time of tasks by Little's law. Indeed, the summand on the right-hand side is the total number of tasks at servers with exactly $i$ tasks in steady state.

The following proposition is proved in Appendix \ref{app: proofs of several results} by coupling a simple load balancing process with a single-server queue, and provides a geometric lower bound for $\expect*{q(i)}$.

\begin{proposition}
	\label{prop: simple load balancing processes}
	Let $\bX$ be a simple and ergodic load balancing process with load $\rho$. Then the steady-state occupancy associated with $\bX$ satisfies that
	\begin{equation*}
		\expect*{q(i)} \geq \frac{\left[r(\rho, |S|)\right]^i}{|S|} \quad \text{for all} \quad i \in \N, \quad \text{where} \quad r(\rho, x) \defeq \left(\frac{\rho}{x}\right)^x.
	\end{equation*}
\end{proposition}

This lower bound is coarse for small $i$. However, it shows that $\expect*{q(i)}$ does not decay faster than geometrically as $i$ increases when the number of servers is finite. In this paper we derive similar bounds for general networks, and obtain necessary conditions for them to remain valid in the limit as the network grows large.

\section{Main results}
\label{sec: main results}

Consider a load balancing process $\bX$ given by a bipartite graph $G = (D, S, E)$ and rate functions $\map{\lambda}{D}{(0, \infty)}$ and $\map{\mu}{S}{(0, \infty)}$. Also, fix $\lambda_0$ and $\mu_0$ such that:
\begin{equation}
	\label{eq: rate bounds}
	0 < \lambda_0 \leq \min_{d \in D} \lambda(d) \quad \text{and} \quad \max_{u \in S} \mu(u) \leq \mu_0 < \infty, \quad \text{and let} \quad \rho_0 \defeq \frac{\lambda_0}{\mu_0}.
\end{equation}

We denote the neighborhood and degree of a dispatcher $d$ by
\begin{equation*}
	\calN(d) \defeq \set{u \in S}{(d, u) \in E} \quad \text{and} \quad \deg(d) \defeq \left|\calN(d)\right|,
\end{equation*}
respectively, and use analogous notations for the servers. Our main results provide lower bounds for the mean steady-state occupancy in terms of the following metrics.

\begin{definition}
	\label{def: graph metrics}
	The \emph{flexibility metrics} for $G = (D, S, E)$ are
	\begin{equation*}
		 \alpha_G \defeq \frac{1}{|S|}\sum_{u \in S} \min \set{\deg(d)}{d \in \calN(u)} \quad \text{and} \quad \beta_G \defeq \frac{1}{|D|}\sum_{d \in D} \deg(d).
	\end{equation*}
\end{definition}

Intuitively, the degree of a dispatcher measures its flexibility for distributing incoming tasks across different servers, which helps to maintain shorter queues. The metric $\alpha_G$ is the average over the servers of the degree of the least flexible dispatcher to which each server is connected, whereas $\beta_G$ is just the average degree over the dispatchers. The former metric is defined from the perspective of each server considering the compatible dispatcher that tends to be the main source of congestion at the server, whereas the latter metric considers the perspective of the dispatchers and measures the average flexibility. For the bipartite graph of a simple load balancing process, $\alpha_G = \beta_G = |S|$. Hence, our first main result generalizes Proposition \ref{prop: simple load balancing processes} from simple to arbitrary load balancing processes.

\begin{theorem}
	\label{the: first lower bound}
	Suppose that $\bX$ is ergodic and let $q$ be the steady-state occupancy. Then
	\begin{equation*}
		\expect*{q(i)} \geq \frac{\left[r(\rho_0, \alpha_G)\right]^i}{\alpha_G} \quad \text{for all} \quad i \geq \frac{1}{\rho_0},
	\end{equation*}
	where $r$ is defined as in Proposition \ref{prop: simple load balancing processes}.
\end{theorem}

The proof of the theorem is provided in Section \ref{sec: first lower bound}. In fact, we establish that the same bound holds more generally, with $\alpha_G$ replaced by the following expression:
\begin{equation*}
	\theta_G \defeq \frac{1}{|S|}\sum_{u \in S} \sum_{d \in \calN(u)} \theta(d, u) \deg(d) \geq \alpha_G,
\end{equation*}
where $\map{\theta}{D \times S}{[0, 1]}$ can be any function such that $\theta(d, u)$ sums one over $d \in \calN(u)$ for each fixed $u \in S$, i.e., the inner sum in the expression in the middle is a convex combination of degrees. We prove that the largest lower bound is achieved when we have equality above, which corresponds to $\theta(d, u) = 0$ if $\deg(d) > \min \set{\deg(e)}{e \in \calN(u)}$.

Our second main result provides a geometric lower bound for the steady-state occupancy which is similar to that in Theorem \ref{the: first lower bound} but depends on $\beta_G$ instead of $\alpha_G$.

\begin{theorem}
	\label{the: second lower bound}
	Suppose that $\bX$ is ergodic and let $q$ be the steady-sate occupancy. Then
	\begin{equation*}
		\expect*{q(i)} \geq \frac{\rho_0}{\beta_G(\beta_G + 1) + \rho_0} \frac{\left[r(\rho_0, \beta_G + 1)\right]^i}{\beta_G + 1} \quad \text{for all} \quad i \geq \frac{1}{\rho_0},
	\end{equation*}
	where $r$ is defined as in Proposition \ref{prop: simple load balancing processes}.
\end{theorem}

This theorem is proved in Section \ref{sec: second lower bound} and complements Theorem \ref{the: first lower bound}. Indeed, the following example shows that $\alpha_G / \beta_G$ can be arbitrarily small or large, and we prove in Section \ref{sec: first lower bound} that for $\rho_0$ and $i$ fixed, the function $x \mapsto [r(\rho_0, x)]^i / x$ is decreasing over the interval $[\rho_0, \infty)$. This implies that the lower bound provided in Theorem \ref{the: first lower bound} can be much larger than that given in Theorem \ref{the: second lower bound}, and that the opposite situation is possible as well.

\begin{example}
	\label{ex: flexibility metrics} The bipartite graphs $G_n^1$ and $G_n^2$ depicted in Figure \ref{fig: examples} satisfy that:
	\begin{equation*}
		\alpha_{G_n^1} = 1, \quad \beta_{G_n^1} = \frac{n + 1}{2}, \quad \alpha_{G_n^2} = \frac{n + 1}{2} \quad \text{and} \quad \beta_{G_n^2} = \frac{2n}{n + 1}. 
	\end{equation*}
	Therefore, $\alpha_{G_n^1} / \beta_{G_n^1} \to 0$ and $\alpha_{G_n^2} / \beta_{G_n^2} \to \infty$ as $n \to \infty$. In both graphs a significant fraction of the dispatchers is compatible with just one server, and thus has little flexibility. This is better captured by $\alpha_{G_n^1}$ in the first graph and by $\beta_{G_n^2}$ in the second graph.
\end{example}

\begin{figure}
	\centering
	\begin{subfigure}{0.49\columnwidth}
		\centering
		\includegraphics{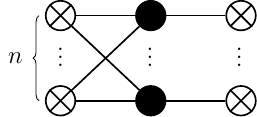}
		\subcaption{Bipartite graph $G_n^1$}
	\end{subfigure}
	\hfill
	\begin{subfigure}{0.49\columnwidth}
		\centering
		\includegraphics{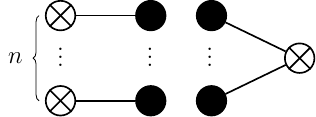}
		\subcaption{Bipartite graph $G_n^2$}
	\end{subfigure}
	\caption{Dispatchers and servers are represented by crossed white circles and black circles, respectively. Each server in $G_n^1$ is connected to all the dispatchers at its left and only one dispatcher at its right. The connected component on the right of $G_n^2$ consists of one dispatcher connected to $n$ servers, whereas each of the other $n$ connected components consists of one dispatcher and one server.}
	\label{fig: examples}
\end{figure}

An important consequence of Theorems \ref{the: first lower bound} and \ref{the: second lower bound} is the following result, which concerns a growing sequence of processing networks and provides a geometric lower bound for the limiting steady-state occupancy when the flexibility metrics remain bounded.

\begin{theorem}
	\label{the: sequences of graphs}
	Consider a sequence of bipartite graphs $G_n = (D_n, S_n, E_n)$ with
	\begin{equation*}
		\alpha \defeq \liminf_{n \to \infty} \alpha_{G_n} \quad \text{and} \quad \beta \defeq \liminf_{n \to \infty} \beta_{G_n}.
	\end{equation*}
	Also, fix sequences of rate functions $\map{\lambda_n}{D_n}{(0, \infty)}$ and $\map{\mu_n}{S_n}{(0, \infty)}$ such that
	\begin{equation*}
		\liminf_{n \to \infty} \min_{d \in D_n} \lambda_n(d) > \lambda_0 > 0 \quad \text{and} \quad \limsup_{n \to \infty} \max_{u \in S_n} \mu_n(u) < \mu_0 < \infty.
	\end{equation*}
	Suppose that the load balancing processes associated with these bipartite graphs and rate functions are ergodic and let $\rho_0 \defeq \lambda_0 / \mu_0$. Then the steady-state occupancies satisfy:
	\begin{equation*}
		\liminf_{n \to \infty} \expect*{q_n(i)} \geq \max\left\{\frac{\left[r(\rho_0, \alpha)\right]^i}{\alpha}, \frac{\rho_0}{\beta(\beta + 1) + \rho_0}\frac{\left[r(\rho_0, \beta + 1)\right]^i}{\beta + 1}\right\} \quad \text{for all} \quad i \geq \frac{1}{\rho_0}.
	\end{equation*} 
\end{theorem}

\begin{proof}
	By assumption, there exists $n_0 \geq 1$ such that
	\begin{equation*}
		\min_{d \in D_n} \lambda_n(d) > \lambda_0 \quad \text{and} \quad \max_{u \in S_n} \mu_n(u) < \mu_0 \quad \text{for all} \quad n \geq n_0.
	\end{equation*}
	As a result, Theorems \ref{the: first lower bound} and \ref{the: second lower bound} imply that if $n \geq n_0$, then
	\begin{equation*}
		\expect*{q_n(i)} \geq \max\left\{\frac{\left[r(\rho_0, \alpha_{G_n})\right]^i}{\alpha_{G_n}}, \frac{\rho_0}{\beta_{G_n}\left(\beta_{G_n} + 1\right) + \rho_0}\frac{\left[r(\rho_0, \beta_{G_n} + 1)\right]^i}{\beta_{G_n} + 1}\right\} \quad \text{for all} \quad i \geq \frac{1}{\rho_0}.
	\end{equation*}
	Since $x \mapsto [r(\rho_0, x)]^i / x$ is continuous in $(0, \infty)$ and $\min\{\alpha_{G_n}, \beta_{G_n}\} \geq 1$, we get
	\begin{equation*}
		\liminf_{n \to \infty} \expect*{q_n(i)} \geq \max\left\{\frac{\left[r(\rho_0, \alpha)\right]^i}{\alpha}, \frac{\rho_0}{\beta\left(\beta + 1\right) + \rho_0}\frac{\left[r(\rho_0, \beta + 1)\right]^i}{\beta + 1}\right\} \quad \text{for all} \quad i \geq \frac{1}{\rho_0},
	\end{equation*}
	which completes the proof.
\end{proof}

If $\alpha$ or $\beta$ are finite, then this result implies that the mean steady-state occupancy does not decay faster than geometrically in the limit, which gives a partial converse for results in \cite{weng2020boptimal,rutten2022load,rutten2024mean,zhao2022exploiting,zhao2023optimal, mukherjee2018asymptotically, budhiraja2019supermarket}. Loosely speaking, these papers consider sequences of networks with suitable connectivity properties and prove mean-field limits showing that the steady-state occupancy behaves as if the bipartite graphs where complete in the limit. The connectivity properties imply that $\beta = \infty$ and the limiting steady-state occupancy decays faster than geometrically. Theorem \ref{the: sequences of graphs} implies that such mean-field limits are not possible if $\beta < \infty$.

\section{Monotone transformations}
\label{sec: monotone transformations}

In this section we define transformations of bipartite graphs and rate functions such that the load balancing process associated with the transformed bipartite graph and rate functions has shorter queues in a stochastic sense, a property that plays an important role in the proof of Theorem \ref{the: first lower bound}. We remark that the transformations considered here where first introduced in \cite{goldsztajn2024server} and that the monotonicity property was proved there.

Let $\bX_1$ be a load balancing process associated with a bipartite graph $G_1 = (D_1, S_1, E_1)$ and rate functions $\map{\lambda_1}{D_1}{(0, \infty)}$ and $\map{\mu_1}{S_1}{(0, \infty)}$. One of the transformations that we consider involves coupling the potential departure processes of certain servers, and we need to apply this transformation multiple times to prove Theorem \ref{the: first lower bound}. We thus assume that some servers may have the same potential departure process and associate with $\bX_1$ a partition $\calS_1$ of $S_1$ such that all the servers in $U \in \calS_1$ have the same potential departure process. Specifically, the servers in $U$ have potential departures at the jump times of some common Poisson process and a potential departure leads to an actual departure if the server is not idle. Clearly, we must have $\mu_1(u) = \mu_1(v)$ if $u, v \in U$ and $U \in \calS_1$.

The bipartite graph and rate functions obtained after some given transformation are denoted by $G_2 = (D_2, S_2, E_2)$, $\map{\lambda_2}{D_2}{(0, \infty)}$ and $\map{\mu_2}{S_2}{(0, \infty)}$. The partition of $S_2$ indicating the servers with a common potential departure process is denoted by $\calS_2$ and the associated load balancing process is denoted by $\bX_2$.

\begin{definition}
	\label{def: monotone transformations}
	We consider the following transformations.
	\begin{itemize}
		\item \textit{Arrival rate decrease.} The arrival rate of tasks is decreased for some dispatchers. Specifically, $\lambda_1(d) \geq \lambda_2(d)$ for all $d \in D_1$ while $G_2 \defeq G_1$, $\calS_2 \defeq \calS_1$ and $\mu_2 \defeq \mu_1$.
		
		\item \textit{Service rate increase.} The service rate of tasks is increased for some servers. Namely, $\mu_1(u) \leq \mu_2(u)$ for all $u \in S_1$ while $G_2 \defeq G_1$, $\calS_2 \defeq \calS_1$ and $\lambda_2 \defeq \lambda_1$.
		
		\item \textit{Edge simplification.} A compatibility relation $(d, u) \in E_1$ is removed while a server $v \notin S_1$ and the compatibility relation $(d, v)$ are incorporated. Specifically,
		\begin{equation*}
			D_2 \defeq D_1, \quad S_2 \defeq S_1 \cup \{v\} \quad \text{and} \quad E_2 \defeq \left(E_1 \setminus \left\{(d, u)\right\}\right) \cup \left\{(d, v)\right\}.
		\end{equation*}
		The potential departure process of $v$ is the same as for $u$. Namely, suppose that $U$ is the element of the partition $\calS_1$ such that $u \in U$. Then we let
		\begin{equation*}
			\calS_2 \defeq \left(\calS_1 \setminus \left\{U\right\}\right) \cup \left\{U \cup \{v\}\right\} \quad \text{and} \quad \mu_2(v) \defeq \mu_1(u).
		\end{equation*}
		Further, $\lambda_2(e) \defeq \lambda_1(e)$ for all $e \in D_2$ and $\mu_2(w) \defeq \mu_1(w)$ for all $w \in S_1$.
	\end{itemize}
\end{definition}

\begin{remark}
	\label{rem: edge simplification that creates isolated servers}
	Suppose that $u \in S_1$ is only compatible with $d \in D_1$ and consider the edge simplification that removes $(d, u)$ and adds the server $v$. Then $u$ is an isolated server in $G_2$ and the process $\set{\bX_2(w)}{u \neq w \in S_2}$ is statistically identical to $\bX_1$. In this particular situation, the edge simplification, as defined above, creates an isolated server in $G_2$, but the rest of this network behaves exactly as the original network $G_1$. In order to preserve the standing assumption that the bipartite graphs do not have isolated nodes, we adopt the convention that the edge simplification is the identity transformation when $\deg(u) = 1$, i.e., $G_2 \defeq G_1$, $\calS_2 \defeq \calS_1$, $\lambda_2 \defeq \lambda_1$, $\mu_2 \defeq \mu_1$, and in particular $\bX_2 \defeq \bX_1$.
\end{remark}

An illustrative example of the edge simplification transformation is depicted in Figure \ref{fig: edge simplification}. Using this transformation cleverly, we may obtain load balancing processes that are easier to analyze. Particularly, a server connected to several dispatchers can be transformed into several servers connected to different dispatchers by applying several edge simplifications. Combined with the monotonicity property stated next, this provides a powerful tool for stochastically bounding the queue lengths of the servers from below. This property is a direct extension of the results in \cite[Section 7.2]{goldsztajn2024server} and is proved in Appendix \ref{app: proofs of several results}.

\begin{figure}
	\centering
	\begin{subfigure}{0.49\columnwidth}
		\centering
		\includegraphics{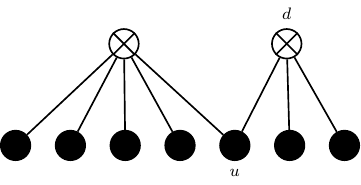}
		\caption{Bipartite graph $G_1$, before edge simplification}
	\end{subfigure}%
	\hfill
	\begin{subfigure}{0.49\columnwidth}
		\centering
		\includegraphics{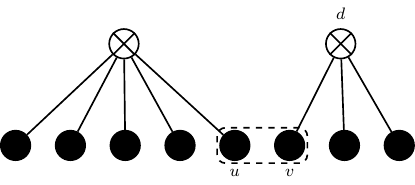}
		\caption{Bipartite graph $G_2$, after edge simplification}
	\end{subfigure}
	\caption{Edge simplification that removes the compatibility relation $(d, u)$ and incorporates server $v$ and the compatibility relation $(d, v)$. The servers $u$ and $v$ have the same potential departure process.}
	\label{fig: edge simplification}
\end{figure}

\begin{proposition}
	\label{prop: monotonicity property}
	Suppose that $\bX_2$ is obtained from $\bX_1$ through arrival rate decrease or service rate increase transformations, and that the random vectors $\bX_1(0)$ and $\bX_2(0)$ which describe the initial conditions are identically distributed. Then
	\begin{equation}
		\label{eq: monotonicity for rate modifications}
		\cprob*{\bX_1(t, u) \geq i} \geq \cprob*{\bX_2(t, u) \geq i} \quad \text{for all} \quad t \geq 0, \quad u \in S_1 \quad \text{and} \quad i \in \N.
	\end{equation}
	Suppose now that $\bX_2$ is obtained from $\bX_1$ by means of the edge simplification that removes the compatibility relation $(d, u)$ and incorporates server $v$. Assume also that $\bX_1(0)$ and $\sett{\bX_2(0, w)}{w \in S_1}$ are identically distributed and $\bX_2(0, u) \geq \bX_2(0, v)$ with probability one. Then the following inequalities hold:
	\begin{equation}
		\label{eq: monotonicity for edge simplification}
		\cprob*{\bX_1(t, u) \geq i} \geq \cprob*{\bX_2(t, v) \geq i} \quad \text{and} \quad \cprob*{\bX_1(t, w) \geq i} \geq \cprob*{\bX_2(t, w) \geq i}
	\end{equation}
	for all $t \geq 0$, $w \in S_1$ and $i \in \N$. In addition, if $\bX_1$ is ergodic, then $\bX_2$ is ergodic in either case. Furthermore, then \eqref{eq: monotonicity for rate modifications} and \eqref{eq: monotonicity for edge simplification} hold for the stationary distributions.
\end{proposition}

\section{Proof of Theorem \ref{the: first lower bound}}
\label{sec: first lower bound}

In this section we establish Theorem \ref{the: first lower bound}. For this purpose, we consider the load balancing process $\bX$ of Section \ref{sec: main results} and assume that it is ergodic with steady-state occupancy $q$.

\begin{figure}
	\centering
	\begin{subfigure}{0.49\columnwidth}
		\centering
		\includegraphics{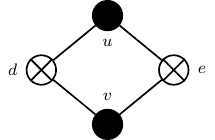}
		\caption{Bipartite graph $G$, before edge simplifications}
	\end{subfigure}%
	\hfill
	\begin{subfigure}{0.49\columnwidth}
		\centering
		\includegraphics{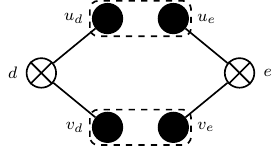}
		\caption{Bipartite graph $G_0$, after edge simplifications}
	\end{subfigure}
	\caption{Bipartite graph $G_0$ obtained after performing an edge simplification at each edge of $G$. Each of the sets of servers $\{u_d, u_e\}$ and $\{v_d, v_e\}$ has a common potential departure process.}
	\label{fig: full edge simplification}
\end{figure}

We perform an edge simplification at all the edges of the bipartite graph $G$, proceeding sequentially and recalling the convention in Remark \ref{rem: edge simplification that creates isolated servers}; see Figure \ref{fig: full edge simplification} for a small example. From these transformations we get the bipartite graph $G_0 = (D_0, S_0, E_0)$ given by
\begin{equation*}
D_0 \defeq D, \quad S_0 \defeq \set{u_d}{(d, u) \in E} \quad \text{and} \quad E_0 \defeq \set{(d, u_d)}{(d, u) \in E},
\end{equation*}
and the sets of servers with common potential departure processes are given by
\begin{equation*}
	\calS_0 \defeq \set{\set{u_d}{d \in \calN(u)}}{u \in S};
\end{equation*}
in this section the notations $\calN$ and $\deg$ always refer to $G$ and not $G_0$.

In other words, observe that the graph $G_0$ has multiple connected components, each consisting of a single dispatcher $d$ connected to several servers $u_d$. The degree in $G_0$ of each dispatcher $d \in D_0 = D$ is the same as in $G$, whereas each server in $G_0$ has degree one. In addition, two servers of the forms $u_d$ and $u_e$ have the same potential departure process, but note that these servers are always in different connected components of $G_0$.

Consider the load balancing process $\bX_0$ given by $G_0$ and the constant rate functions such that all the dispatchers have arrival rate $\lambda_0$ and all the servers have service rate $\mu_0$. It follows from \eqref{eq: rate bounds} that we may obtain $\bX_0$ from $\bX$ by applying an arrival rate decrease to all the dispatchers, a service rate increase to all the servers and multiple edge simplifications. By applying Proposition \ref{prop: monotonicity property} multiple times, we conclude that
\begin{equation}
	\label{eq: stochastic dominance after full edge simplification}
	\cprob*{X(u) \geq i} \geq \cprob*{X_0(u_d) \geq i} \quad \text{for all} \quad (d, u) \in E \quad \text{and} \quad i \in \N,
\end{equation}
where $X$ and $X_0$ are the stationary distributions of $\bX$ and $\bX_0$, respectively. Recall that $\bX$ is ergodic by assumption. Thus, $\bX_0$ is ergodic too by Proposition \ref{prop: monotonicity property}.

Let $\map{\theta}{D \times S}{[0, 1]}$ be any function such that
\begin{equation*}
	\sum_{d \in \calN(u)} \theta(d, u) = 1 \quad \text{for all} \quad u \in S,
\end{equation*}
and recall that $q$ is the steady-state occupancy for $\bX$. By \eqref{eq: stochastic dominance after full edge simplification} and the above property of $\theta$,
\begin{equation*}
	\expect*{q(i)} = \frac{1}{|S|}\sum_{u \in S} \cprob*{X(u) \geq i} \geq \frac{1}{|S|}\sum_{u \in S} \sum_{d \in \calN(u)} \theta(d, u) \cprob*{X_0(u_d) \geq i} \quad \text{for all} \quad i \in \N.
\end{equation*}
The servers $\set{u_d}{u \in \calN(d)}$ and dispatcher $d$ form a connected component of $G_0$ for all $d \in D$, and the potential departure processes of these servers are mutually independent. It follows that $\sett{X_0(u_d)}{u \in \calN(d)}$ is the stationary distribution of a simple load balancing process. Since the aforementioned connected component has arrival rate $\lambda_0$ and $\deg(d)$ servers with service rate $\mu_0$, we conclude from Proposition \ref{prop: simple load balancing processes} that
\begin{equation}
	\label{eq: lower bound in terms of convex combinations}
	\expect*{q(i)} \geq \frac{1}{|S|}\sum_{u \in S} \sum_{d \in \calN(u)} \theta(d, u) \frac{\left[r\left(\rho_0, \deg(d)\right)\right]^i}{\deg(d)},
\end{equation}

In order to complete the proof of Theorem \ref{the: first lower bound}, we need the following technical lemma.

\begin{lemma}
	\label{lem: properties of r}
	Fix constants $\rho > 0$ and $k \geq 1 / \rho$, and consider the function defined as
	\begin{equation*}
		f(x) \defeq \frac{\left[r(\rho, x)\right]^k}{x} = \frac{1}{x}\left(\frac{\rho}{x}\right)^{kx} \quad \text{for all} \quad x > 0.
	\end{equation*}
	This function is strictly decreasing and convex in $[\rho, \infty).$
\end{lemma}

The lemma is proved in Appendix \ref{app: proofs of several results} and implies that
\begin{equation}
	\label{eq: lower bound in terms of theta}
	\expect*{q(i)} \geq \frac{\left[r(\rho_0, \theta_G)\right]^i}{\theta_G} \quad \text{for all} \quad i \geq \frac{1}{\rho_0} \quad \text{with} \quad \theta_G \defeq \frac{1}{|S|}\sum_{u \in S} \sum_{d \in \calN(u)} \theta(d, u) \deg(d).
\end{equation}
Indeed, observe that $\deg(d) > \rho_0$ for all $d \in D$ since otherwise the simple load balancing process $\sett{\bX_0(u_d)}{u \in \calN(d)}$ would not be ergodic, contradicting Proposition \ref{prop: monotonicity property}. Since
\begin{equation*}
	\sum_{u \in S} \sum_{d \in \calN(u)} \frac{\theta(d, u)}{|S|} = 1,
\end{equation*}
the convexity property in Lemma \ref{lem: properties of r} gives \eqref{eq: lower bound in terms of theta}. Taking any function $\tilde{\theta}$ such that $\tilde{\theta}(d, u) = 0$ if $\deg(d) > \min \set{\deg(e)}{e \in \calN(u)}$, we obtain $\tilde{\theta}_G = \alpha_G$, and thus
\begin{equation*}
	\expect*{q(i)} \geq \frac{\left[r(\rho_0, \alpha_G)\right]^i}{\alpha_G} \geq \frac{\left[r(\rho_0, \theta_G)\right]^i}{\theta_G} \quad \text{for all} \quad i \geq \frac{1}{\rho_0}
\end{equation*}
and all functions $\theta$. Indeed, the second inequality follows from Lemma \ref{lem: properties of r} and $\rho_0 \leq \alpha_G \leq \theta_G$ for all $\theta$. This completes the proof of Theorem \ref{the: first lower bound}.

\section{Proof of Theorem \ref{the: second lower bound}}
\label{sec: second lower bound}

In this section we prove Theorem \ref{the: second lower bound}. For this purpose, we fix $\gamma > \beta_G$ and define:
\begin{equation*}
	D_\gamma \defeq \set{d \in D}{\deg(d) < \gamma} \subset D \quad \text{and} \quad S_\gamma \defeq \bigcup_{d \in D_\gamma} \calN(d) \subset S.
\end{equation*}
Consider the bipartite graph $G_0 = (D_0, S_0, E_0)$ obtained by applying an edge simplification to each edge in $D_\gamma^c \times S_\gamma$; we use the notations $D_\gamma^c \defeq D \setminus D_\gamma$ and $S_\gamma^c \defeq S \setminus S_\gamma$. Consider also the bipartite graph $G_\gamma = (D_\gamma, S_\gamma, E_\gamma)$ such that $E_\gamma \defeq E \cap (D_\gamma \times S_\gamma)$, which is both a subgraph of $G$ and $G_0$. Furthermore, observe that $G_\gamma$ is isolated from $G_0 \setminus G_\gamma$ and that the potential departure processes of the servers in $S_\gamma$ remain mutually independent after the edge simplifications. Figure \ref{fig: graph structure} depicts the three bipartite graphs.

\begin{figure}
	\centering
	\begin{subfigure}{0.49\columnwidth}
		\centering
		\includegraphics{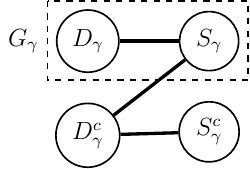}
		\subcaption{Bipartite graph $G$}
	\end{subfigure}
	\hfill
	\begin{subfigure}{0.49\columnwidth}
		\centering
		\includegraphics{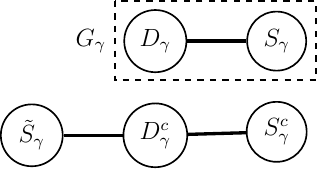}
		\subcaption{Bipartite graph $G_0$}
	\end{subfigure}
	\caption{Schematic view of the bipartite graphs $G$, $G_0$ and $G_\gamma$. The circles represent sets of nodes and a thick line between two sets of nodes indicates that there may exist edges between the two sets. The set $\tilde{S}_\gamma$ is the set of servers incorporated through the edge simplifications. The potential departure processes of servers in $S_\gamma$ are coupled with those of servers in $\tilde{S}_\gamma$ but are mutually independent over $S_\gamma$.}
	\label{fig: graph structure}
\end{figure}

Let $\bX_0$ and $\bX_\gamma$ be the load balancing processes associated with the bipartite graphs $G_0$ and $G_\gamma$, respectively, for the constant rate functions such that all the dispatchers have arrival rate $\lambda_0$ and all the servers have service rate $\mu_0$. It follows from \eqref{eq: rate bounds}, the ergodicity of $\bX$ and Proposition \ref{prop: monotonicity property} that $\bX_0$ is ergodic. The processes $\sett{\bX_0(u)}{u \in S_\gamma}$ and $\bX_\gamma$ are continuous-time Markov chains with the same transition rates since $G_\gamma$ and $G_0 \setminus G_\gamma$ are disjoint, and in particular $\bX_\gamma$ is ergodic as well. If we let $X$, $X_0$ and $X_\gamma$ be the stationary distributions of $\bX$, $\bX_0$ and $\bX_\gamma$, respectively, then Proposition \ref{prop: monotonicity property} further implies that
\begin{equation*}
	\cprob*{X(u) \geq i} \geq \cprob*{X_0(u) \geq i} = \cprob*{X_\gamma(u) \geq i} \quad \text{for all} \quad i \geq 0 \quad \text{and} \quad u \in S_\gamma.
\end{equation*}

Let $q_\gamma$ denote the steady-state occupancy for $\bX_\gamma$ and recall that the potential departure processes of the servers in $S_\gamma$ remain mutually independent after the edge simplifications. Therefore, we may apply Theorem \ref{the: first lower bound} to $\bX_\gamma$, which yields:
\begin{equation*}
	\frac{1}{|S_\gamma|}\sum_{u \in S_\gamma} \cprob{X(u) \geq i} \geq \frac{1}{|S_\gamma|}\sum_{u \in S_\gamma} \cprob{X_\gamma(u) \geq i} = \expect*{q_\gamma(i)} \geq \frac{\left[r(\rho_0, \alpha_{G_\gamma})\right]^i}{\alpha_{G_\gamma}}
\end{equation*}
for all $i \geq 1 / \rho_0$. Moreover, the degree in $D_\gamma$ is upper bounded by $\ceil{\gamma - 1}$, i.e., the smallest integer larger than or equal to $\gamma - 1$. Thus, $\alpha_{G_\gamma} \leq \ceil{\gamma - 1}$ and Lemma \ref{lem: properties of r} implies that
\begin{equation}
	\label{eq: lower bound using ceiling}
	\frac{1}{|S_\gamma|}\sum_{u \in S_\gamma} \cprob{X(u) \geq i} \geq \frac{\left[r(\rho_0, \ceil{\gamma - 1})\right]^i}{\ceil{\gamma - 1}} \quad \text{for all} \quad i \geq \frac{1}{\rho_0}.
\end{equation}

Applying Markov's inequality to the uniform distribution over $D$, we obtain
\begin{equation*}
	\frac{|D_\gamma^c|}{|D|} = \frac{1}{|D|}\sum_{d \in D} \ind{\deg(d) \geq \gamma} \leq \frac{\beta_G}{\gamma} \quad \text{and hence} \quad \frac{|D|}{|D_\gamma|} \leq \frac{\gamma}{\gamma - \beta_G}.
\end{equation*}
Since $\bX$ is ergodic, \eqref{eq: ergodicity condition} and \eqref{eq: rate bounds} imply that $\rho_0 |D_\gamma| \leq |S_\gamma|$. Also,
\begin{equation*}
	\left|S_\gamma^c\right| \leq \left|E \cap \left(D \times S_\gamma^c\right)\right| \leq \left|E \cap \left(D_\gamma^c \times S\right)\right| = \sum_{d \in D_\gamma^c} \deg(d) \leq \beta_G|D|.
\end{equation*}
The first inequality uses the standing assumption that no sever is isolated, and the second inequality holds since $E \cap (D_\gamma \times S_\gamma^c) = \emptyset$ by definition of $S_\gamma$. Putting all together, we get
\begin{equation*}
	\frac{|S_\gamma^c|}{|S_\gamma|} \leq \frac{\beta_G|D|}{\rho_0|D_\gamma|} \leq \frac{\beta_G \gamma}{\rho_0(\gamma - \beta_G)}.
\end{equation*}

Recall that $q$ is the steady-state occupancy for $\bX$. By \eqref{eq: lower bound using ceiling} and the above inequality,
\begin{align*}
	\expect*{q(i)} &\geq \frac{1}{|S|}\sum_{u \in S_\gamma} \cprob*{X(u) \geq i} \\
	&= \frac{|S_\gamma|}{|S|}\frac{1}{|S_\gamma|}\sum_{u \in S_\gamma} \cprob*{X(u) \geq i} \\
	&\geq \frac{|S_\gamma|}{|S_\gamma| + |S_\gamma^c|}\frac{\left[r(\rho_0, \ceil{\gamma - 1})\right]^i}{\ceil{\gamma - 1}} \geq \frac{\rho_0(\gamma - \beta_G)}{\rho_0(\gamma - \beta_G) + \beta_G\gamma}\frac{\left[r(\rho_0, \ceil{\gamma - 1})\right]^i}{\ceil{\gamma - 1}}
\end{align*}
for all $i \geq 1 / \rho_0$ and $\gamma > \beta_G$. If we take $\gamma = \beta_G + 1$, then
\begin{equation*}
	\expect*{q(i)} \geq \frac{\rho_0}{\beta_G (\beta_G + 1) + \rho_0}\frac{\left[r(\rho_0, \ceil{\beta_G})\right]^i}{\ceil{\beta_G}} \geq \frac{\rho_0}{\beta_G (\beta_G + 1) + \rho_0}\frac{\left[r(\rho_0, \beta_G + 1)\right]^i}{\beta_G + 1}
\end{equation*}
for all $i \geq  1 / \rho_0$, where the last inequality holds by Lemma \ref{lem: properties of r} since $\rho_0 \leq \ceil{\beta_G} \leq \beta_G + 1$. This completes the proof of Theorem \ref{the: second lower bound}.

\section{Conclusion}
\label{sec: conclusion}

We have derived geometric lower bounds for the expected steady-state occupancy in processing networks, which are based on novel metrics that capture the degree of flexibility for distributing incoming tasks over the network. Furthermore, we have established that the geometric decay of the mean steady-state occupancy cannot be avoided even in the large-scale limit unless the flexibility metrics diverge. In particular, diverging flexibility is a necessary condition for mean-field limits in the literature showing that growing processing networks perform asymptotically as the classic Power-of-$d$ or JSQ policies. Thus, flexibility must always be considered in the design of processing networks at scale.


\begin{appendices}
	
\section{Additional proofs}
\label{app: proofs of several results}

\begin{proof}[Proof of Proposition \ref{prop: simple load balancing processes}]
	Let $\bY$ be the number of tasks in a single-server queue with Poisson arrivals at rate $\lambda$ and exponentially distributed service times with rate $\mu |S|$. The proof is based on coupling the stochastic processes $\bX$ and $\bY$ as follows.
	
	Suppose that both systems have the same arrival process and complete tasks according to potential departure processes, i.e., a server completes a task if and only if it has a potential departure while its queue is not empty. Some server in the simple load balancing system has a potential departure if and only if the single-server queue has a potential departure, and each server of the simple load balancing system is equally likely to have a potential departure when the single-server queue has a potential departure. In particular, the potential departure processes of the servers in the simple load balancing system are thinnings of the potential departure process associated with the single-server queue. It is clear that this construction does not change the laws of $\bX$ and $\bY$.
	
	It follows from the above construction that
	\begin{equation*}
		\sum_{u \in S} \bX(t, u) \geq \bY(t) \quad \text{for all} \quad t \geq 0 \quad \text{if this holds for} \quad t = 0;
	\end{equation*}
	this property holds for all sample paths of $(\bX, \bY)$. Because $\bX$ is ergodic, $\lambda < \mu |S|$ and thus $\bY$ is ergodic as well. We conclude that the stationary distributions $X$ and $Y$ of the continuous-time Markov chains $\bX$ and $\bY$, respectively, exist, are unique and satisfy:
	\begin{equation*}
		\cprob*{\sum_{u \in S} X(u) \geq i} \geq \cprob*{Y \geq i} \quad \text{for all} \quad i \in \N.
	\end{equation*}
	
	The above inequality implies that
	\begin{equation*}
		P\left(Y \geq |S|i\right) \leq P\left(\sum_{u \in S} X(u) \geq |S|i\right) \leq P\left(\bigcup_{u \in S} \left\{X(u) \geq i\right\}\right) \leq |S|\cprob*{X(v) \geq i},
	\end{equation*}
	where $v \in S$ can be any server. Recalling that $\rho = \lambda / \mu$, we obtain:
	\begin{equation*}
		\cprob*{X(u) \geq i} \geq \frac{1}{|S|}P\left(Y \geq |S|i\right) = \frac{1}{|S|}\left(\frac{\rho}{|S|}\right)^{|S|i} = \frac{\left[r(\rho, |S|)\right]^i}{|S|} \quad \text{for all} \quad u \in S.
	\end{equation*}
	The claim follows directly from this inequality.
\end{proof}

\begin{proof}[Proof of Proposition \ref{prop: monotonicity property}]
	The inequalities \eqref{eq: monotonicity for rate modifications} and \eqref{eq: monotonicity for edge simplification} are proved in \cite[Section 7.2]{goldsztajn2024server}. For the claim about the stationary distributions, observe that \eqref{eq: monotonicity for rate modifications} and \eqref{eq: monotonicity for edge simplification} imply that the state where all the queues are empty is positive recurrent for $\bX_2$ if this holds for $\bX_1$. Since both processes are irreducible, $\bX_2$ is ergodic if $\bX_1$ is. Suppose that both processes are ergodic and denote their stationary distributions by $X_1$ and $X_2$, respectively, then
	\begin{equation*}
		\cprob*{X_j(u) \geq i} = \lim_{T \to \infty} \frac{1}{T}\int_0^T \cprob*{\bX_j(t, u) \geq i}dt \quad \text{for} \quad j \in \{1, 2\}, \quad u \in S_j \quad \text{and} \quad i \in \N.
	\end{equation*}
	Therefore, \eqref{eq: monotonicity for rate modifications} and \eqref{eq: monotonicity for edge simplification} also hold if we replace $\bX_j(t, u)$ by $X_j(u)$ for $j \in \{1, 2\}$.
\end{proof}

\begin{proof}[Proof of Lemma \ref{lem: properties of r}]
	First observe that
	\begin{equation*}
		f'(x) = f(x)g(x) \quad \text{with} \quad g(x) \defeq k\log \rho - k \log x - k - \frac{1}{x} \quad \text{for all} \quad x > 0.
	\end{equation*}
	Clearly, $g(\rho) < 0$. Also, $g$ is nonincreasing in $[\rho, \infty)$ since $\rho \geq 1 / k$ by assumption and
	\begin{equation*}
		g'(x) = \frac{1}{x^2} - \frac{k}{x} = \frac{1}{x}\left(\frac{1}{x} - k\right) \quad \text{for all} \quad x > 0.
	\end{equation*}
	Hence, $g(x) < 0$ for all $x \geq \rho$, and we conclude that $f$ is strictly decreasing in $[\rho, \infty)$.
	
	Now observe that the second derivative of $f$ is given by
	\begin{align*}
		f''(x) &= f(x)g'(x) + f(x)\left[g(x)\right]^2 \\
		&= \left[\frac{1}{x^2} - \frac{k}{x} + \left(k\log x - c\right)^2 + \frac{2\left(k\log x - c\right)}{x} + \frac{1}{x^2}\right]f(x) \quad \text{for all} \quad x > 0,
	\end{align*} 
	where we write $c \defeq k\log \rho - k$ for brevity. In particular, we obtain:
	\begin{equation*}
		\frac{x^2f''(x)}{f(x)} = \left(k\log x - c\right)^2x^2 + \left(2k\log x - 2c - k\right)x + 2 \quad \text{for all} \quad x > 0.
	\end{equation*}
	
	Suppose that $f''(x_0) = 0$ for some $x_0 > 0$, and consider the functions defined as:
	\begin{equation*}
		a(x) \defeq \left(k\log x - c\right)^2 \quad \text{and} \quad b(x) \defeq 2k\log(x) - 2c - k \quad \text{for all} \quad x > 0.
	\end{equation*}
	If $a(x_0) = 0$, then $x_0 = \e^{c / k} = \rho / \e$ and $b(x_0) = -k$. Therefore,
	\begin{equation*}
		\frac{\rho}{\e} = x_0 = \frac{-2}{b(x_0)} = \frac{2}{k}.
	\end{equation*}
	We conclude that $f''(x_0) = 0 = a(x_0)$ can only occur if $k \rho = 2 \e$ and $x_0 = \rho / \e < \rho$.
	
	Consider now the situation where $f''(x_0) = 0 < a(x_0)$. Then the quadratic polynomial $p(x) \defeq a(x_0) x^2 + b(x_0)x + 2$ has a root at $x_0$.  Assuming that $b(x_0) \geq 0$, we get
	\begin{equation*}
		0 < x_0 \leq \frac{-b(x_0) + \sqrt{\left[b(x_0)\right]^2 - 8a(x_0)}}{2a(x_0)} \leq 0,
	\end{equation*}
	which is a contradiction. Therefore, we must have
	\begin{equation*}
		2k\log\left(\frac{x_0}{\rho}\right) + k = b(x_0) < 0,
	\end{equation*}
	and this implies that $x_0 < \rho / \sqrt{e} < \rho$.
	
	The above arguments establish that $f''$ does not have any roots in $[\rho, \infty)$. In addition, $x^2 f''(x) / f(x) \to +\infty$ as $x \to +\infty$, which implies that $f''$ is positive for all large enough $x$. Since $f''$ is continous, we conclude that $f''$ is positive in $[\rho, \infty)$ and thus $f$ is convex.
\end{proof}

\end{appendices}
	
\newcommand{\noop}[1]{}
\bibliographystyle{IEEEtranS}
\bibliography{bibliography}

\begin{thebibliography}{10}
\providecommand{\url}[1]{#1}
\csname url@samestyle\endcsname
\providecommand{\newblock}{\relax}
\providecommand{\bibinfo}[2]{#2}
\providecommand{\BIBentrySTDinterwordspacing}{\spaceskip=0pt\relax}
\providecommand{\BIBentryALTinterwordstretchfactor}{4}
\providecommand{\BIBentryALTinterwordspacing}{\spaceskip=\fontdimen2\font plus
\BIBentryALTinterwordstretchfactor\fontdimen3\font minus
  \fontdimen4\font\relax}
\providecommand{\BIBforeignlanguage}[2]{{%
\expandafter\ifx\csname l@#1\endcsname\relax
\typeout{** WARNING: IEEEtranS.bst: No hyphenation pattern has been}%
\typeout{** loaded for the language `#1'. Using the pattern for}%
\typeout{** the default language instead.}%
\else
\language=\csname l@#1\endcsname
\fi
#2}}
\providecommand{\BIBdecl}{\relax}
\BIBdecl

\bibitem{van2018scalable}
M.~{\noop{Boor}}van~der Boor, S.~C. Borst, J.~S.~H. van Leeuwaarden, and
  D.~Mukherjee, ``Scalable load balancing in networked systems: A survey of
  recent advances,'' \emph{SIAM Review}, vol.~64, no.~3, pp. 554--622, 2022.

\bibitem{budhiraja2019supermarket}
A.~Budhiraja, D.~Mukherjee, and R.~Wu, ``Supermarket model on graphs,''
  \emph{The Annals of Applied Probability}, vol.~29, no.~3, pp. 1740--1777,
  2019.

\bibitem{foss1998stability}
S.~G. Foss and N.~I. Chernova, ``On the stability of a partially accessible
  multi-station queue with state-dependent routing,'' \emph{Queueing Systems},
  vol.~29, pp. 55--73, 1998.

\bibitem{gast2015power}
N.~Gast, ``The power of two choices on graphs: the pair-approximation is
  accurate?'' \emph{ACM SIGMETRICS Performance Evaluation Review}, vol.~43,
  no.~2, pp. 69--71, 2015.

\bibitem{goldsztajn2024server}
D.~Goldsztajn, S.~C. Borst, and J.~S. Van~Leeuwaarden, ``Server saturation in
  skewed networks,'' \emph{Proceedings of the ACM on Measurement and Analysis
  of Computing Systems}, vol.~8, no.~2, pp. 1--37, 2024.

\bibitem{lu2011join}
Y.~Lu, Q.~Xie, G.~Kliot, A.~Geller, J.~R. Larus, and A.~Greenberg,
  ``{Join-Idle-Queue}: A novel load balancing algorithm for dynamically
  scalable web services,'' \emph{Performance Evaluation}, vol.~68, no.~11, pp.
  1056--1071, 2011.

\bibitem{menich1991optimality}
R.~Menich and R.~F. Serfozo, ``Optimality of routing and servicing in dependent
  parallel processing systems,'' \emph{Queueing Systems}, vol.~9, pp. 403--418,
  1991.

\bibitem{mitzenmacher2001power}
M.~Mitzenmacher, ``The power of two choices in randomized load balancing,''
  \emph{IEEE Transactions on Parallel and Distributed Systems}, vol.~12,
  no.~10, pp. 1094--1104, 2001.

\bibitem{mukherjee2018asymptotically}
D.~Mukherjee, S.~C. Borst, and J.~S.~H. van Leeuwaarden, ``Asymptotically
  optimal load balancing topologies,'' \emph{Proceedings of the ACM on
  Measurement and Analysis of Computing Systems}, vol.~2, no.~1, pp. 1--29,
  2018.

\bibitem{mukherjee2018universality}
D.~Mukherjee, S.~C. Borst, J.~S.~H. van Leeuwaarden, and P.~A. Whiting,
  ``Universality of power-of-$d$ load balancing in many-server systems,''
  \emph{Stochastic Systems}, vol.~8, no.~4, pp. 265--292, 2018.

\bibitem{rutten2022load}
D.~Rutten and D.~Mukherjee, ``Load balancing under strict compatibility
  constraints,'' \emph{Mathematics of Operations Research}, vol.~48, no.~1, pp.
  227--256, 2023.

\bibitem{rutten2024mean}
------, ``Mean-field analysis for load balancing on spatial graphs,'' \emph{The
  Annals of Applied Probability}, vol.~34, no.~6, pp. 5228--5257, 2024.

\bibitem{sparaggis1993extremal}
P.~D. Sparaggis, D.~Towsley, and C.~Cassandras, ``Extremal properties of the
  shortest/longest non-full queue policies in finite-capacity systems with
  state-dependent service rates,'' \emph{Journal of Applied Probability},
  vol.~30, pp. 223--236, 1993.

\bibitem{stolyar2015pull}
A.~L. Stolyar, ``Pull-based load distribution in large-scale heterogeneous
  service systems,'' \emph{Queueing Systems}, vol.~80, pp. 341--361, 2015.

\bibitem{vvedenskaya1996queueing}
N.~D. Vvedenskaya, R.~L. Dobrushin, and F.~I. Karpelevich, ``Queueing system
  with selection of the shortest of two queues: An asymptotic approach,''
  \emph{Problemy Peredachi Informatsii}, vol.~32, no.~1, pp. 20--34, 1996.

\bibitem{weng2020boptimal}
W.~Weng, X.~Zhou, and R.~Srikant, ``Optimal load balancing with locality
  constraints,'' \emph{Proceedings of the ACM on Measurement and Analysis of
  Computing Systems}, vol.~4, no.~3, pp. 1--37, 2020.

\bibitem{zhao2023optimal}
Z.~Zhao and D.~Mukherjee, ``Optimal rate-matrix pruning for large-scale
  heterogeneous systems,'' \emph{Queueing Systems}, vol. 110, no.~1, p.~16,
  2026.

\bibitem{zhao2022exploiting}
Z.~Zhao, D.~Mukherjee, and R.~Wu, ``Exploiting data locality to improve
  performance of heterogeneous server clusters,'' \emph{Stochastic Systems},
  vol.~14, no.~3, pp. 229--272, 2024.

\end{thebibliography}
	
\end{document}